\documentclass[11pt]{article}
\usepackage{amssymb, verbatim, amsmath}
\usepackage[all]{xy}
\usepackage{amsthm}
\usepackage{amssymb, amsmath}
\parskip \baselineskip
\usepackage{graphicx,color}

\newcommand{\bc}{\begin{center}}
\newcommand{\ec}{\end{center}}
\newcommand{\be}{\begin{enumerate}}
\newcommand{\ee}{\end{enumerate}}
\newcommand{\beq}{\begin{equation}}
\newcommand{\eeq}{\end{equation}}
\newcommand{\bi}{\begin{itemize}}
\newcommand{\ei}{\end{itemize}}
\newcommand{\bd}{\begin{description}}
\newcommand{\ed}{\end{description}}
\newcommand{\ba}{\begin{array}}
\newcommand{\bea}{\begin{eqnarray*}}
\newcommand{\eea}{\end{eqnarray*}}
\newcommand{\ea}{\end{array}}
\newcommand{\bt}{\begin{tabular}}
\newcommand{\et}{\end{tabular}}

\newcommand{\bmi}{\begin{minipage}}
\newcommand{\emi}{\end{minipage}}

\newcommand{\lb}{\linebreak}

\newcommand{\Z}{\Bbb Z}

\newcommand{\R}{\Bbb R}

\newtheorem{theorem}{Theorem}[section]

\newtheorem{corollary}[theorem]{Corollary}

\newtheorem{definition}[theorem]{Definition}
\newtheorem{example}[theorem]{Example}

\newtheorem{lemma}[theorem]{Lemma}

\newtheorem{proposition}[theorem]{Proposition}
\newtheorem{remark}[theorem]{Remark}

\begin{document}

\bc {\bf\large Toward an intuitive understanding of the structure of near-vector spaces}\\[3mm]
{\sc K-T Howell \& S Marques$^{*}$}

\it\small
Department of Mathematical Sciences,
Stellenbosch University,
Stellenbosch, 7600,\lb
South Africa\\
\rm e-mail: kthowell@sun.ac.za \& smarques@sun.ac.za
\ec

\quotation{\small {\bf Abstract:} In this paper we analyse the definition Andr\'e proposed for near-vector spaces to make it more transparent. We also study the class of near-vector spaces over division rings and give a characterisation of regularity that gives a new insight into the decomposition of near-vector spaces into regular subspaces. We explicitly describe span and deduce a new characterisation of subspaces.}

\normalsize

\section{Introduction}
The notion of a structure with less linearity than a traditional vector space has been studied by a number of authors. First Beidleman \cite{Beidleman} used near-ring modules to construct a near-vector space; whereas Andr\'e \cite{Andre} used an additive group together with a set of endomorphisms of the group, satisfying certain conditions. Next Karzel \cite{Karzel} defined a near-vector space structure mimicking a vector space structure but without the scalars acting as endomorphisms on the underlying group. 

In the first part of the paper we analyse Andr\'e's definition and in doing so we attempt to highlight how Andr\'e's definition seems to be the most suitable and natural definition to work with since it allows a lot of flexibility in the structure. We state a structural Lemma \ref{nearfielddecomp} for general vector spaces that hopefully reveals the connection between vector spaces and near-vector spaces and how the generalisation affects the structure. We also study one of the main new features of near-vector spaces, namely the non-unique additive structures on the underlying multiplicative group. In particular, we state the Key Lemma \ref{basisaddition} for this paper that give a necessary condition for two additive structure to be the same. Several examples are exhibited to highlight the special features of near-vector spaces.

In the second part of the paper we focus on constructions of near-vector spaces over division rings. We first focus on regularity. The main result (Theorem \ref{vstheorem}) of this section fully characterises the uniqueness of the induced additive structure into regularity, division ring and quasi-kernel structure. This permitted us to give an alternate proof of the decomposition into regular subspaces when the underlying set can be endowed with a division ring structure. We then focus on the structure of Span. Our main result in that section, Theorem \ref{spanstruc}, describes precisely the span of any one element set in $V$. Surprisingly, it reveals that $Span$ and linear combinations agree in that setting as they do in classical linear algebra. Generalising this result of Span to a general set, we also deduce a very useful characterisation of subspaces which actually corresponds to the usual characterisation in classical linear algebra while the initial definition required a difficult check for the generating condition.


\section{Preliminary material}
In this section we define concepts that are analogous to those that are central to traditional linear algebra.

In \cite{Andre} the notion of a near-vector space was defined as: 

\begin{definition}\label{defnvs}(\cite{Andre}, Definition 4.1, p.9)
A non-trivial near-vector space is a pair $(V,A)$ which satisfies the following conditions:
\begin{enumerate}
\item $(V,+)$ is a group and $A$ is a set of endomorphisms of $V$;
\item A contains the endomorphisms $0$, $id$ and $-id$;
\item $A^{*}=A\backslash\{0\}$ is a subgroup of the group Aut$(V)$;
\item If $\alpha x=\beta x$ with $x\in V$ and $\alpha,\beta\in A$, then $\alpha=\beta$ or $x=0$, i.e. $A$ acts fixed point free on $V$;
\item The quasi-kernel $Q(V)$ of $V$, generates $V$ as a group. Here, $$Q(V)=\{x\in V|\forall \alpha, \beta\in A,\,\exists \gamma\in A \text{ such that }\alpha x+\beta x=\gamma x\}.$$
\end{enumerate}
\end{definition}
Note that the trivial near-vector space has to be considered separately since it results in $A^{*}$ being empty. We will write $Q(V)^*$ for $Q(V) \backslash \{0\}$ throughout this paper and just $Q$ if it does not cause confusion. Also, we note that we write scalars on the left and and as a result, make use of left near-fields. We will use $F$ to denote a near-field and $F_d$ its distributive elements. For more on near-fields we refer the reader to \cite{Mel} and \cite{Pilz}.

From Andr\'e's definition, it is natural to define the concept of a homomorphism as follows:

\begin{definition}(\cite{HowMey}, Definition 3.2, p.57)\label{homo}
We say that two near-vector spaces $(V_1,A_1)$ and $(V_2,A_2)$ are homomorphic (written $(V_1,A_1)\cong (V_2,A_2)$) if there are group homomorphisms
 $\theta : (V_1,+) \rightarrow (V_2,+)$ and
 $\eta : (A_1^*,\cdot) \rightarrow (A_2^*,\cdot)$
such that $\theta(\alpha x) = \eta(\alpha)\theta(x)$ for all $x\in V_1$ and $\alpha\in A_1^*$. We will write homomorphisms as pairs $(\theta,\eta).$
\end{definition}
This permits us to compare near-vector spaces.

We will also need the notion of a subspace and span as they will have an important role to play as in traditional linear algebra.

\begin{definition}\label{def2}(\cite{Howell2}, Definition 2.3, p.3)
 If $(V,A)$ is a near-vector space and $\emptyset \neq V'\subseteq V$ is such that $V'$ is the subgroup of $(V, +)$ generated additively by $AX = \{ax\,|\, x \in X, a \in A\}$, where $X$ is an independent subset of $Q(V)$, then we say that $(V', A)$ is a subspace of $(V, A)$, or simply $V'$ is a subspace of $V$ if $A$ is clear from the context.  
\end{definition}

\begin{definition}(\cite{HowSanSpan}, Definition 3.2, p.3235)
Let $(V,A)$ be a near-vector space, then the span of a set $S$ of vectors is defined to be the intersection $W$ of all subspaces of $V$ that contain $S,$ denoted $\mbox{span }S$. 
\end{definition}

It is straightforward to verify that $W$ is a subspace, called the subspace spanned by $S,$ or conversely, $S$ is called a spanning set of $W$ and we say that $S$ spans $W$. Moreover, if we define $\mbox{span }\emptyset = \{0\}$, then it is not difficult to check that $\mbox{span }S$ is the set of all possible linear combinations of $S$ if $S\subseteq Q(V).$ However; if $S$ contains elements outside of the quasi-kernel then it is not clear that these two coincide (See Theorem  \ref{spanstruc} ).  For this reason we define:

\begin{definition}
Let $V$ be a near-vector space. 
For every $v\in V$,  we define the linear combinations of $v$ as the set 
$$L(v) =  \{ \alpha_1 (v) + \cdots + \alpha_t (v) | t \in \mathbb{N} \ and \ \alpha_i \in A, for \ i \in \{ 1, \ldots , t\} \}.$$  
\end{definition}

In a near-vector space linear independence is defined in terms of the elements of the quasi- kernel $Q(V).$

\begin{definition}(\cite{Andre}, p.302) 
Let $(V,A)$ be a near-vector space. We say that a set $S \subseteq Q(V)$ is linearly independent if for any $v_1, \cdots , v_n$ in $S$ and $\alpha_1 , \cdots, \alpha_n\in A$ such that if
$$ \alpha_1 v_1 + \cdots + \alpha_n v_n =0$$ 
then 
$$\alpha_1 = \cdots = \alpha_n =0.$$ 
Otherwise we say they are linearly dependent. 
\end{definition}

\begin{definition}(\cite{Andre}, p.303)
Let $(V,A)$ be a near-vector space, then an independent generating set for $Q(V)$ is called a basis of $V$ and its cardinality is called the dimension of $V$.
\end{definition}

As for a vector space, one can prove that a near-vector space has a basis by showing that from the set of elements of $Q(V)$ generating $V$, one can always extract a basis. Thus any near-vector space admits a basis in $Q(V).$ It is routine as in linear algebra to prove that there is a well defined notion of dimension, i.e. if a near-vector space has a finite basis all the bases have the same number of elements. See \cite{Andre} for more details on this and a proof that any near-vector space admits a basis by enlarging an existing linear independent set. 

The dimension of an element is defined as follows:

\begin{definition}(\cite{HowSanSpan}, Definition 3.5, p.3236)
For $v\in V\setminus\{0\}$ we define the dimension of $v$ to be 
$$n=\min\left\lbrace m\in \mathbb{N}\mid v=\sum_{i=1}^m \alpha_iu_i,\text{ with }u_i\in Q(V)\setminus\{0\},\alpha_i\in A\setminus\{0\},i=1,\ldots,m  \right\rbrace ,$$
we denote it by $dim(v)=n$ and $dim(v)=0$ if $v$ is the zero vector.
\end{definition}

The concept of regularity is a central notion in the study of near-vector spaces. Andr\'e called the regular spaces the building blocks of near-vector space theory. They happen to be well-behaved, as we will see. This led to the Decomposition Theorem, where any near-vector space is decomposed into regular parts. See Theorem 4.13, p. 3.6 in \cite{Andre}.

\begin{definition}\label{def4} (\cite{Andre} Definition 4.7, p.11) 
A near-vector space is {\it regular} if any two vectors of  $Q(V)\backslash\{0\}$ are {\it compatible, }i.e. if for any two vectors $u$ and $v$ of $Q(V)$ there exists a $\lambda \in A\backslash\{0\}$ such that $u + v\lambda \in Q(V)$. 
\end{definition}

Note that every near vector space $(V, A)$ with $dim ( V) \leq 1,$ is regular and has $Q(V)=V$.

The addition in $V$ naturally gives rise to an addition in $A$ as follows: 

\begin{definition}(\cite{Andre}, p.299)
Let $(V,A)$ be a near-vector space and let $v \in Q(V) \backslash \{0\}$. Define the operation $+_{v}$ on A by  
$$
(\alpha +_{v} \beta)v := \alpha v  + \beta  v   \; (\alpha, \beta \in A).
$$
\end{definition}

With this addition, $(A,+_v,\cdot)$ is a near-field (see \cite{Andre}). The essentiality of this definition will become clear to the reader in Section 3 and 4.

This addition gives rise to:

\begin{definition}(\cite{Andre}, Definition 2.6, p.301)
Let $(V,A)$ be a near-vector space and let $u \in Q(V) \backslash \{0\}$. Define the \it kernel \rm $R_{u}(V) = R_{u}$ of   
$(V,A)$ by the set
\[R_{u} := \{v \in V \,|\, (\alpha +_{u} \beta)v = \alpha v  + \beta v \mbox{ for every } \alpha, \beta \in A\}.\]
\end{definition}

In Theorem \ref{vstheorem} in Section 4, we will see how the notion of a kernel relates to regularity.

\section{Analysis of the definition of Near-vector spaces}

Definition \ref{defnvs} might be a bit disconcerting to some. As a result, we decided to revisit the definition and try to understand how essential each assumption is to a good notion of what a near-vector space should be as the natural way to  widely generalise a vector space after the introduction of near-fields. As part of this we will later see how the concepts of subspace, span and linear combinations interplay.  Traditionally in linear algebra, a vector space is a field together with an abelian group endowed with an endomorphism action by the field on the abelian group, called the scalar multiplication. A near-vector space structure does not just result in a weakening of one of the distributive laws, but allows more generality.  In the definition we do not begin by fixing a near-field, but instead fix a multiplicative group. From Andr\'e's definition (Definition \ref{defnvs}), one can construct a near-field such that the near-vector space can be viewed in the expected way mentioned above. Nevertheless, the choice of the near-field is not unique, as we will see later on. As result of a lack of uniqueness of the underlying near-field, the direct generalisation of the traditional construction is weaker than what is proposed by Andr\'e. 

\medskip

The essence of the geometry behind linear algebra stems from the existence of a coordinate system. The geometry resides in the notion of a basis which gives rise to unit vectors and coordinate axes which are one-dimensional subspaces. In order to reproduce this idea in a more general setting, we start with an additive group $V$ and a set $A$ such that there exist elements $v \in V$ such that $Av$ represents a set of coordinate axes and $A$ becomes a near-field induced by $(V,+)$ that is in  bijection with $Av$. 

Cleverly, Andr\'e noted that one does not need this property to hold for every $v\in V$ to guarantee the existence of a coordinate system, one only needs to ensure that such $v$'s generate $V$. As we will see this is guaranteed with the property that
$$Q(V) =  \{ v\in V | L(v) = Av\}=\{v \in V |Av = span(v)\}$$ 
generates $V$ (see Definition \label{defAndre}, 5.). 
Note that, for $v \in Q(V)$, we have 
$$ L(v) = span(v)=Av.$$
Thus picking $v \in Q(V)$ guarantees that $Av = \mbox{span }v = L(v).$ This shows that the coordinate axes $Av$ are now subspaces of $V,$ which we would expect. The fact that $Q(V)$ generates $V$ allows us to see that $$V =span (\{ v \in V | L( v )= Av\}).$$ The notion of a basis will give us access to a coordinate system from some $Av$ with $v \in Q(V).$

\medskip
We are still expecting an underlying near-field structure. We will now explain how this structure can be revealed without being fixed in advance. 
To define a well-defined operation on $A$ induced from that of $V$, we need for all $\alpha, \beta \in A$ that $\alpha v + \beta v = \gamma v \in Av,$ for a unique $\gamma \in A.$ The existence of $\gamma$ is equivalent to $v$ belonging to $Q(V)$ and its uniqueness is guaranteed by requiring fixed point freeness (see Definition \label{defnvs}, 4.). As a consequence of the existence and uniqueness of $\gamma,$ if we fix a nonzero $v$ in $Q(V),$ then for any $\alpha, \beta \in A,$ we can define an operation $+_{v}$ on $A$ that sets $\alpha +_{v} \beta $ to be this $\gamma.$  

\medskip

What is left to show now is that the addition of $V$ naturally induces a structure of a near-field on $A.$
 This mimics what we have for traditional vector spaces, where the underlying structure would be a field. The difference is that the underlying near-field structure is not fixed beforehand and non unique. 
 
We will use $0_A$ and $0_V$ to denote the identities of $A$ and $V,$ respectively. In order for the group structure of $V$ to induce a group structure on $A,$ we need that for all $v \in Q(V), 0_{A} \cdot v = 0_{V}.$ Therefore $0_A$ acts as the zero endomorphism on $V$ which is generated by elements of $Q(V).$ In order to have a meaningful endomorphism action of $A$ on $V$ that induces a near-field structure on $A$, we will also need a multiplicative structure on $A^{*} = A \backslash \{0\}.$

We will use $1_A$ to denote the multiplicative identity of $A$. In order to ensure $A$ acts as an endomorphism on $V$ we will need $1_A$ to act as the identity endomorphism $Id$. 

If $v \in Q(V),$ then
$$1_{A}v - 1_{A}v = 0_{V},$$ by the group structure of $(V,+),$ thus $((1_A)+_{v}(-1_{A}))v = 0_{V}.$ Now by the fixed point free property $(1_A)+_{v}(-1_{A}) = 0_A.$ Thus $-1_{A}$ is the inverse of $1_A,$ so that for all $v \in V,$ $(-1_{A})(v) = -v$ and $-1_A$ act as $-Id$ in $V$. 

Moreover, since $A$ induces an endomorphism action on $V$, if we use $-a$ to denote the additive inverse of $a\in A$ we then have 
$$ (-1_A \cdot a ) v = (-1_A) (av) = - (av).$$ 
As before we can prove that $-a=-1_A(a)$ is the additive inverse of $A$. So that $a\in A$, implies $-a\in A$. 

{\it Note that as a consequence of $-Id$ acting as an endomorphism of $(V, +)$ we have that $- (v + w)= -w-v = -v -w$, for all $v, w \in V$ therefore $(V, +)$ must be an abelian group. }

\medskip

Let us summarise what has been identified. As in traditional linear algebra we start with a group $(V, +)$ and a set of endomorphisms of $V$, $(A, \cdot).$  

To have access to a coordinate system (basis) formed from an underlying near-field structure so that we are able to study  the geometric properties of the system, we need 
\begin{enumerate} 
\item A set of coordinate axes $Av$ that generate $(V, +)$ as an additive group.  This is guaranteed by the property "The quasi-kernel $Q(V)$ of $V$, generates $V$ as a group. " in Definition \ref{defnvs}. 
\item That $(V, +)$ induces a near-field structure   on $(A, \cdot)$, more precisely there exists a group operation $+'$ on $A$ induced by the operation $+$ on $V$ such that $(A, +' , \cdot )$ is a near-field. The properties: "A contains the endomorphisms $0$, $id$ and $-id$", "$A^{*}=A\backslash\{0\}$ is a subgroup of the group Aut$(V)$;" and the fixed point free property precisely ensures that $(A, +_v, \cdot)$ is a near- field for any  $v \in Q(V)\backslash\{0\}$. 
\end{enumerate}

\begin{remark}
We note that Andr\'e's definition (Definition \ref{defnvs} 2.) requires the existence of elements in $A$ that acts as $Id_A$ and $-Id_A$ on $V$, that will imply that $A$ will contain an element of multiplicative order $2$. 
By Cauchy and Lagrange's Theorem, $(A^{*},\cdot)$ is a group with even order if and only if there is a $x\neq 1_{A} \in A^{*}$ such that $x^2 =1_{A}$. For any $v \in V,$ we have that $x ( v + xv) =  v+ xv$ since $V$ is abelian and by the fixed point free property, since $x\neq 1_{A},$ we have that $v+xv=0$ and $xv=-v$. When the characteristic of $V$ is not $2,$ this element $x$ acts as $-Id$ in $A.$ However, when the characteristic of $V$ is $2$, $x$ and $1_{A}$ have the same action, contradicting the fixed point freeness. To conclude, when the characteristic of $V$ is not $2$, $A^{*}$ will have exactly one element of order $2$ while in characteristic $2,$ $A^{*}$ cannot have any element of order $2.$ Note that if $A^{*}$ is finite will imply that if the characteristic of $V$ is $2,$ $A^{*}$ will have even order, while if the characteristic is not $2,$ the order of $A^{*}$ has to be odd.
\end{remark}

From linear algebra the most basic example of a vector space is a field over itself, hence it would be essential to have that the additive group of a near-field be a near-vector space over its multiplicative group. 

\begin{example}(\cite{Andre})
Let $F$ be a near-field that is not $M_{C}({\Z}_{2})$ (see \cite{Pilz}, Proposition 8.1, p.249), then we have that:

\begin{enumerate}
\item $(F,+)$ is a group and $(F,\cdot) $ is a set of endomorphisms of $F$;
\item $0_F$ acts as an endomorphism by assumption since $F$ is zero-symmetric. Clearly $1_F$ is an endomorphism while $-1_F \in F$ and is an endormorphism by Proposition 8.10, p. 151 in \cite{Pilz};
\item It is clear that $F^{*}=F\backslash\{0\}$ is a subgroup of the group Aut$(F)$;
\item $F$ acts fixed point free on itself since $F$ has multiplicative inverses;
\item The quasi-kernel is the set $Q(F) = \{(k_i)\lambda \,|\, \lambda \in , k_i \in F_{d}\} = F$ (Theorem 4.4 p.304 in \cite{Andre}). 
For any nonzero $x \in F,$ $\{x\}$ is a basis of $(F,F)$ of dimension 1.
\end{enumerate}
Thus $(F,F)$ is a near-vector space.
\end{example}

In order to state the next lemma, we need some definitions.

\begin{definition} For $I$ an index set, possibility infinite, we define
$$A^{(I)}= \{ a : I \rightarrow A| a \text{ is zero for all but finitely many } i \in I \},$$
and the standard basis elements $e_i \in A^{(I)}$ for any $i \in I$ defined on $I$ as 
$$e_i(j) = \delta_{i,j} =\left\{ \begin{array}{cc} 0 & \text{ if $i\neq j$} \\ 1 & \text{ otherwise.} \end{array}\right. $$
\end{definition}
The following lemma gives a way to visually compare near-field theory to classical linear algebra, emphasising the structure that completes the analysis above. 

\begin{lemma} \label{nearfielddecomp}
The following assertions are equivalent: 
\begin{enumerate}
\item  $(V, A)$ is a near-vector space. 
\item $V \simeq^{\phi} A^{(I)},$ $e_i\in Q(A^{(I)})$ and the additive structure of the near-vector space $(A^{(I)}, A)$ is given by $+'$ defined for any $a, b \in A^{(I)} $ point-wise by
$$ (a+'b)(i)= a(i) +_{\phi^{-1}(e_i)} b(i)$$
and the scalar multiplication induced by the action of $A$ on $V.$ 
\end{enumerate}
\end{lemma} 
\begin{proof} 
$$1. \Rightarrow 2.$$ 

We know that $V$ admits a basis $\{v_1, \cdots , v_n\}$. It is not hard to prove then that
$$V\simeq \oplus_{i \in I} Av_i.$$

Then, we obtain the isomorphism $\phi$ by composing with the isomorphism
$$\psi:   \oplus_{i \in I} Av_i \rightarrow A^{(I)}$$ 
sending $v_i$ to $e_i$ and extended by linearity by 
sending any element in $ \oplus_{i \in I} Av_i $,  $\sum_{i \in I}  \alpha_{i} v_i$ to $\sum_{i \in I} \alpha_i e_i $. 

The scalar multiplication on $A^{(I)}$ is induced by the action of $A$ on $V$ component-wise via the bijections $Av_i \simeq A$ for any $i\in I$ which sends $\alpha v_i$ to $\alpha$ for $i \in I.$ We note that this is well-defined by the fixed point free property. 

It is not difficult to prove that this will define a near- vector space isomorphism for the near-vector space structure identified in the statement. The converse is also clear.
\end{proof} 
\begin{remark} We suppose that the  scalar multiplication on $A^{(I)}$ is induced by the given multiplicative structure of $A$ as follows:
$\cdot '$ is defined point-wise for any $\alpha \in A $ and $ a \in A^{(I)}$ by: 
$$ (\alpha \cdot' a ) (i ) = \eta_{i}(\alpha ) a(i),$$ 
where $\eta_i$ is a surjective map to ensure that any element of $A$  acts on every component so that the entire action of $A$ is transmitted to $A^{(I)}.$
Then $\eta_i$ is a near-field isomorphism from $(A,+_{b},\cdot)$ to $(A,+,\cdot),$ where $b \in Q (A^{(I)})\backslash\{0\}.$

Indeed, the injectivity is guaranteed by the fixed point freeness. The multiplicativity follows from the action of $A.$

Moreover, let $b\in Q(A^{(I)})$ nonzero, for all $\alpha , \beta \in A$, we have point-wise
$$(\alpha \cdot' b + \beta \cdot' b)(i) = \eta_i(\alpha) b (i) + \eta_i (\beta) b(i) =(\eta_i(\alpha)+ \eta_i (\beta)) b(i).$$ 

In the other hand 

$$ \alpha \cdot' b + \beta \cdot' b = (\alpha +_b  \beta) \cdot' b = \eta_i (\alpha +_b \beta) b(i).$$

So that by the point free property
$$\eta_i( \alpha) + \eta_i(\beta) = \eta_i (\alpha +_b \beta),$$ 
as stated.

When $+_a = +_b=+$ for any $a, b \in Q (A^{(I)})\backslash\{0\}$, then we have a near-field automorphism of $(A, +, \cdot).$
\end{remark} 

To emphasise the importance of $-1 \in A$ and the quasi-kernel generating $V$ beyond the induced structure of near-field in $A$, we give the following two examples:

\begin{example}\label{jacques}
Take 
$$V = \R$$
and 
$$A = \R^{+}\cup \{0\}.$$
Then 
$$Q(V) = \R.$$ 
Note that all the axioms of Definition \ref{defnvs} are satisfied, except that $A$ contains no elements acting as $-Id_A$ in $V$. But despite the fact we cannot naturally obtain the structure of a near-field in $\mathbb{R}^+$ from the field structure of $\mathbb{R}$ as explained above.The set $\{-1,1\}$ is a generating set of $\R,$ but it is not linearly independent since for $x,y \in \mathbb{R},$ both nonzero, $$x + y(-1) = 0$$ implies that $x = y.$ Thus we also do not have a basis. 
\end{example} 

$A$-groups have been studied. They meet all the requirements of Andr\'e's definition, but assumption 5. of Definition \ref{defnvs} (see \cite{Andre} for more on $A$-groups.)  As an illustration we give the following example where only this assumption is not satisfied. It shows that the notion of an $A$-group would not lead to a good notion of what intuitively we could expect a near-vector space to be. 

\begin{example}\label{lotte}
Take 
$$V = \frac{\mathbb{Z}}{3\mathbb{Z}} \oplus \mathbb{Z}$$
and 
$$A = \{-1,0,1\}.$$
Then 
$$Q(V) = \frac{\mathbb{Z}}{ 3 \mathbb{Z}} \oplus \{ 0\}.$$
The only missing assumption in the definition of near vector space is that $Q(V)$ does not generate $V$. 
\end{example}  


The following example illustrates the flexibility that the definition of a homomorphism allows. In the next section we will identify the important properties of near-vector spaces which will explain the unnatural phenomena in the example below.

\begin{example}
We use $+_{3}$ to denote the addition on $\mathbb{R}$ defined by for all $x, y\in \mathbb{R}$ by
$$ x +_{3} y = \sqrt[3]{x^3 + y^3}.$$ 
Then we can show that $\mathbb{R}_{+_3}=(\mathbb{R}, +_{3}, \cdot )$ is a field. 
It is clear that $\mathbb{R}$ is closed under $+_{3},$ $0$ is the identity element, $+_{3}$ is commutative and $-x$ is an inverse for $x$. 
As for distributivity, we have 
$$\alpha ( x +_{3} y ) = \alpha  \sqrt[3]{x^3 + y^3} =  \sqrt[3]{(\alpha x)^3 +(\alpha  y)^3}= ( \alpha x ) +_{3} (\alpha y).$$ 
Note that any odd power could be used to define the addition, giving infinite field structures on $\R$.

Next we prove that the mapping
$$ \begin{array}{cccc} \phi: & (\mathbb{R}, +_{3}, \cdot ) & \rightarrow & (\mathbb{R}, +, \cdot ) \\ 
& x & \mapsto & x^3 
\end{array}$$ 
is a field isomorphism.
It is well-known that $\phi$ is a bijection and for all $x , y \in \mathbb{R},$ 
$$  \phi ( x  +_{3} y) = x^3 + y^3 = \phi (x) + \phi (y).$$ 


We can construct a near-vector space isomorphism from this. Since $\phi$ induces both a multiplicative and additive isomorphism if we define $\cdot_3$ on for all $x, \alpha \in \mathbb{R}$ by $$\alpha \cdot_3 x = \alpha^{3}x = \phi(\alpha) x,$$ 
This homomorphism induces a commutative diagram: 
$$ \xymatrix{ ( \mathbb{R} , \cdot ) \times ( \mathbb{R} , +_3)  \ar[d]_m  \ar[r]^{\phi \times \phi } & \ar[d]^m ( \mathbb{R} , \cdot )  \times ( \mathbb{R} , +)  \\ 
( \mathbb{R} , +_3 ) \ \ar[r]^{\phi } & ( \mathbb{R} , + ) }$$ 
where $m $ sends $(x, y)$ to $xy,$ the usual multiplication of $x$ and $y$ in $\mathbb{R}$. 

Note that $\mathbb{R}_{+_3}$ is a vector space over $\mathbb{R}_{+_3},$ but only a near-vector space over $\R.$ Clearly, $\mathbb{R}$ is a vector space over $\R.$ Even though they are not isomorphic in the traditional vector space sense, $(\phi , \phi) $ is a near-vector space isomorphism that gives more flexibility than traditional linear algebra would have allowed by only fixing the multiplicative structure of the underlying set.
\end{example}

\begin{example}
We consider $\mathbb{Q}(\sqrt{2})$ and $\mathbb{Q}(\sqrt{3})$. These are non- isomorphic fields with isomorphic multiplicative groups with
$$\mathbb{Q}(\sqrt{2})^* =\mathbb{Q}(\sqrt{3})^* = \mathbb{Z}\times \{ \pm 1 \}$$
since the ring of integers of $\mathbb{Q}(\sqrt{2})$ (respectively $\mathbb{Q}(\sqrt{3})$) is  $\mathbb{Z}(\sqrt{2})$ (resp. $\mathbb{Z}(\sqrt{3})$) are unique factorization domains.

Fixing $A = \mathbb{Z}\times \{ \pm 1 \} \cup \{ 0\}$ we have that $(\mathbb{Q}(\sqrt{2}), A) $ and $\mathbb{Q}(\sqrt{3}), A)$ are near-vector spaces
with $(A, +_{\sqrt{2}}, \cdot)$ and $(A, +_{\sqrt{3}}, \cdot)$ non-isomorphic.

This also permits us to construct the near-vector space $(\mathbb{Q}(\sqrt{3})\oplus \mathbb{Q}(\sqrt{2}), A)$
which has no equivalent in traditional linear algebra. 
\end{example}

The different stuctures of the near-field $(A, +_v , \cdot)$ for a near vector space $(V, A)$ are known to be isomorphic when $w \in Av (= Span (V))$, for any $v \in Q(V)\backslash \{ 0 \} $ (Theorem 2.5, p.300 in \cite{Andre}). 

The main results of the section rely upon the following result that holds for any near-vector space and reveals an interesting necessary condition for $+_v = +_w$ for $v, w \in Q(V) \backslash \{ 0 \}$. 

\begin{proposition}[Key Lemma] \label{basisaddition}
Let $(V,A)$ be a near-vector space with $dim(V) >1$ and $Q(V)  = V$ and $S= \{v_i \in Q(V)| i \in I \}$ a linearly independent set (possibly infinite) of $V.$ 

If $v = \sum_{i\in I} \theta_i v_{i}$ and $v' = \sum_{i\in I} \theta_i' v_{i}$ are both in $Q(V)$ where all but finitely many of the $\theta_i$ and $\theta_i'$ are zero and we have that there exist $i_0, j_0 \in I $ with $i_0 \neq j_0$, such that $+_{\theta_{i_0}v_{i_0}}  =+_{\theta'_{j_0}v_{j_0}} $. Then $+_v = +_{v'}= +_{\theta_{i}v_{i}}  =+_{\theta'_{j}v_{j}}$, for all $i,j \in I$. 
\end{proposition}
\begin{proof}
By assumption $v, v' \in Q(V)$, therefore for all $\alpha , \beta \in A$ we get that 
$$\begin{array}{ccl} (\alpha +_{v} \beta) v &=& \alpha ( v) + \beta (v) \\ 
&= & \sum_{i\in I} \alpha \theta_i v_{i} +\sum_{i\in I} \beta  \theta_i v_{i}= \sum_{i\in I} (\alpha \theta_i v_{i} + \beta \theta_i v_{i})  \\ 
&=& \sum_{i\in I} (\alpha +_{\theta_i v_{i}}  \beta) \theta_i  v_{i}. \end{array}$$

Moreover,
$$(\alpha +_{v} \beta) v = \sum_{i\in I } (\alpha +_{v} \beta)\theta_i v_i.$$

Therefore, 
$$ \begin{array}{ccc} 0&=&\sum_{i\in I } (\alpha +_{v} \beta)\theta_i v_i -\sum_{i\in I} (\alpha +_{\theta_i v_{i}}  \beta) \theta_i  v_{i}  \\
&=& \sum_{i\in I } ( (\alpha +_{v} \beta) +_{\theta_i v_i} (-(\alpha +_{\theta_i v_{i}}  \beta))) \theta_i v_i. \end{array}$$

Since $S$ is linearly independent and for any $i \in I$,
$$ (\alpha +_{v} \beta) +_{\theta_i v_i} (- (\alpha +_{\theta_i v_i} \beta ))=0,$$
so we have that $$\alpha +_{v} \beta = \alpha +_{\theta_i v_{i}} \beta. $$ 
For all $\theta_i v_{i}$
(This is deduced from the fixed point freeness of the scalar multiplication which proves that the inverse of $ (\alpha +_{v} \beta)$ for $+$ is $- (\alpha +_{v} \beta).$).
Thus
$$ +_v = +_{\theta_i v_{i}}$$ 
for all $i \in I.$ 
The same can be done to show that $+_{v'} = +_{\theta'_{j} v_{j}}$. 
Proving that 
$$+_{v'} = +_{\theta'_{j_0} v_{j_0}}= +_{\theta_{i_0} v_{i_0}} = +_v.$$
\end{proof}

\begin{remark} 
Let $(V,A)$ be a near-vector space with $dim(V) >1$ and $Q(V)  = V$ and $S= \{v_i \in Q(V)| i \in I \}$ a linearly independent set (possibly infinite) of $V.$ 
If $v = \sum_{i\in I} \theta_i v_{i}\in Q(V)$, then $+_v = +_{\theta_{i}v_{i}}  $, for all $i \in I$. 
\end{remark}

\section{Structural results of Near-vector spaces constructed from division rings}

\subsection{Regular spaces and their decomposition} 
In the next lemma which is not difficult to prove, we discover how the addition $+_v$ that turns $A$ into a near-field for $v\in Q(V) \backslash \{ 0 \}$ results in a special structure on $(A, +_v, \cdot)$.

\begin{lemma}\label{plusmultiples}
Let $V$ be a near-vector space and $v \in Q(V)\{0\}$. 
Then the following are equivalent:
\begin{enumerate} 
\item $+_v= +_{\theta v}$ for all $\theta \in A$;
\item $(A, +_v , \cdot)$ is a division ring. 
\end{enumerate} 
\end{lemma} 
\begin{proof}
It is not difficult to see that:
$$+_v = +_{\theta v}$$ if and only if  $$(\alpha +_v \beta )\theta = (\alpha \theta +_v \beta \theta), \ \forall \alpha, \beta \in A.$$
\end{proof}

The following theorem, attempts to understand the relationship between properties of $Q(V)$, $+_v$ and regularity.

\begin{theorem}\label{vstheorem}
Let $(V, A)$ be a near-vector space. 

The following assumptions are equivalent: 
\begin{enumerate} 
\item For any $v \in Q(V) \backslash \{ 0 \},$ $V$ is a vector space over the near-field $(A ,+_v, \cdot)$ ;
\item There is a $v \in Q(V) \backslash \{ 0 \},$ such that  $V$ is a vector space over the near-field $(A ,+_v, \cdot)$;
\item[1.'] For any $v \in Q(V) \backslash \{ 0 \},$ $V$ is a vector space over the near field $(A ,+_v, \cdot)$ and $(A ,+_v, \cdot)$ is a division ring;
\item[2.']  There is a $v \in Q(V) \backslash \{ 0 \},$ such that  $V$ is a vector space over the near-field $(A ,+_v, \cdot)$ and $(A ,+_v, \cdot)$ is a division ring;
\item[3.] $Q(V)=V$ and $(A, +_v , \cdot )$ is a division ring, for all $v \in Q(V) \backslash \{ 0 \}$. 
\item[4.] $+_v = +_w$ for all $v, w\in Q(V) \backslash \{ 0 \}$. 
\item[5.] $R_{w}(V) = V$ for all $w\in Q(V) \backslash \{ 0 \}$. 
\item[6.] $V$ is regular and for any $v \in V$, $+_v = +_{\theta v}$ for all $v \in Q(V) \backslash \{ 0 \}$ and $\theta \in A$. 
\item[7.] $V$ is regular and for any $v \in Q(V) \backslash \{ 0 \},$ $(A, +_v , \cdot )$ is a division ring. 
\end{enumerate} 

\end{theorem}
\begin{proof} 
$$1. \Leftrightarrow 2. \Leftrightarrow 1.' \Leftrightarrow 2.'$$
is not hard to prove using Lemma \ref{plusmultiples}.

$2. \Rightarrow 3$

If there is $v\in V$ such that $V$  is a vector space over the near-field $(A ,+_v, \cdot)$ then for any $w\in V$ we have 

$$\alpha w + \beta w = (\alpha +_v \beta) w.$$ 

Therefore, $w\in Q(V)$. 

$3. \Rightarrow 4.$

Let $(V,A)$ be a near-vector space such that $Q(V)=V$.  We denote $B = \{ v_i \in Q(V) | i \in I\}$ a basis for $V$. 
For any two nonzero $v, w \in V$, either $v= \theta w$ and $+_v=+_w$, according to Lemma \ref{plusmultiples} since $(A, +_v, \cdot)$ is a division ring. Otherwise we have  $v = \sum_{i\in I}  \theta_i v_{i} \in V$ and $v' = \sum_{i\in I} \theta_i' v_{i} \in V$, and there exist $i\neq j\in I$ such that $\theta_i$ and $\theta_j'$ are nonzero, therefore we can take $w =\theta_i v_{i} + \theta'_{j} v_{j}$, and since by assumption $w\in Q(V)$ we can apply Proposition \ref{basisaddition} to obtain that 
$$+_v =+_w= +_{v'}.$$




$4. \Leftrightarrow 5 \Rightarrow 6. \Leftrightarrow 7.$ is  trivial the last equivalence being a consequence again of Lemma \ref{plusmultiples}.

$6. \Rightarrow 4. $  Suppose that $V$ is regular and for any $v \in V$, $+_v = +_{\theta v}$ for all $v\in V$ and
$\theta \in A$. 

Let $\{ v_i \in Q(V), i \in I \}$ be a basis for $V$. 

For any $i \neq j \in I$,

Then, from the regularity of $V$, there is a $\lambda_{ i,j} \in A$ such that 

$$\omega= v_i + \lambda_{i,j}  v_j \in Q(V).$$

Again applying Proposition \ref{basisaddition}, we get 
$$+_\omega = +_{v_i} = +_{ \lambda_{i,j}  v_j}= +_{v_j}=:+ $$

where the last equality holds since by assumption $(A, +_{v_j} , \cdot)$ is a division ring. 

Let $v \in V$ be nonzero, we have $v= \sum_{i\in I} \theta_i v_{i}.$
Indeed, for any $\alpha , \beta \in A$,
$$ \alpha v + \beta v = \sum_{i\in I}(\alpha \theta_ i v_i + \beta \theta_i v_{i} ) =  \sum_{i\in I} (\alpha \theta_ i  +  \beta \theta_i ) v_{i} =  \sum_{i\in I} (\alpha   +  \beta) \theta_i v_{j_i}=  (\alpha + \beta ) v.$$ 

Therefore $v\in Q(V)$ and $Q(V)=V$. 
\end{proof}

Classical linear algebra is a particular case of a near-vector space corresponding precisely to the regular case, when $(A, +_v , \cdot )$ is a field for every nonzero $v\in Q(V).$

As an application of the previous theorem, 
we reprove the Decomposition Theorem for the division ring case.

\begin{corollary}\label{decomposition}
Let $(V,A)$ be a near-vector space such that for all $v \in Q(V)^{*},$ $(A,+_v,\cdot)$ is a division ring. Then $V$ is the direct sum of regular near-vector spaces ${\cal V}_j$ for $j \in K$ for some index set $K$ such that each $u \in Q(V)^{*}$ lies in precisely one summand ${\cal V}_j$. The subspaces ${\cal V}_j$ are maximal regular near-vector spaces.
\end{corollary}
\begin{proof}
Let ${\cal{B}}= \{v_i \,|\, i \in I\}$ be a canonical basis of $V$ for some index set $J$. We put ${\cal V}_i =\{v \in V \,|\, +_v = +_{v_i}\}.$ We claim that $V = \oplus_{i \in K}{\cal V}_i,$ and that the ${\cal V}_i$ are maximal regular subspaces of $V$ where $K$ is a index set such that for each $i\in J$ there is a unique $k\in K$ such that $v_i \in {\cal V}_k$. For $i \in K,$ the ${\cal V}_i$  are subspaces of $V$ since they are closed under addition by Proposition \ref{basisaddition} and closed under scalar multiplication by Lemma \ref{plusmultiples}. We claim that they are maximal regular subspaces of $V$.  They are regular by Theorem \ref{vstheorem}, since $(A,+_v,\cdot)$ is a division ring and we have that $+_v = +_w$ for all $v,w \in {\cal V}_i.$  Moreover, $Q({\cal V}_i) = {\cal V}_i$ for $i \in K.$  We claim that $V= \oplus_{j \in K} {\cal V}_j.$ %
The sum is clearly direct since we have a basis. It is also clear that $\oplus_{j \in K} {\cal V}_j \subseteq V$. Let $v \in V,$ then since ${\cal B}$ is a basis of $V$, it is clear that $v = \sum_{i\in I} \theta_{i} v_i$ for some $\theta_{i} \in A$ and $v_i \in {\cal B}.$  We set $S_j= \{i\in I | v_i \in {\cal V}_j\}$ for all $j \in K$, then $v= \sum_{j\in K}  w_j$ with $w_j = \sum_{i \in S_j} \theta_i v_i$. In particular, $w_j \in {\cal V}_j$. Thus $v \in \oplus_{j \in K} {\cal V}_j$ and so $V = \oplus_{j \in K} {\cal V}_j.$

To show that each $u \in Q(V)^{*}$ lies in precisely one summand ${\cal V}_j$ we prove that $Q(V) = \uplus_{i=1}^{k} Q({\cal V}_i).$ $\cup_{i=1}^{k} Q({\cal V}_i) \subseteq Q(V)$ is clear. Now let $u \in Q(V),$ then $+_u = +_{u_i}$ for some unique $i \in K$, by Proposition \ref{basisaddition}, since we are supposing that $(A, +_v, \cdot)$ is a division ring for every $v\in Q(V)\backslash\{0\}$.


Finally we need to show the maximality of each ${\cal V}_j$ for $j \in K.$ Suppose that there exists a regular subspace $M$ such that ${\cal V}_j \subset M.$ Then since $M$ is regular, by Theorem \ref{vstheorem}, $+_v = +_w$ for all $v,w \in M.$ Take a $m \in M \backslash {\cal V}_j$ and $v \in {\cal V}_j.$ Then $+_m = +_{v},$ which contradicts our initial partition.  
For the uniqueness of the decomposition, suppose that there are two decompositions $V =  \oplus_{j \in K} {\cal V}_j$ and $V =  \oplus_{s \in R} {\cal W}_s.$ If $v \in {\cal V}_j$ for some $j \in K,$ then by Theorem \ref{vstheorem}, $+_v = +_{v_j}$. But since $V =  \oplus_{s \in R} {\cal W}_s,$  $v_i \in {\cal W}_s$ for some $s \in R$ and again by Theorem \ref{vstheorem}, $+_{v_j} = +_w$ for all $w \in {\cal W}_s.$  Now we have that $+_w = +_{v_i} = +_v,$ so $v \in {\cal W}_s.$
Similarly,  ${\cal W}_s \subseteq {\cal V}_j$.
\end{proof}

\begin{corollary}
If $(V,A)$ is as before and ${\cal{B}}= \{v_i \,|\, i \in I\}$ a canonical basis of $V,$ $K$ a index set such that for each $i\in I$ there is a unique $k\in K$ such that $v_i \in {\cal V}_k$ and we choose $j \in K$.
\begin{itemize}
\item $\{v_i \,|\ v_i \in {\cal V}_j\}$ is a basis for ${\cal V}_j.$ \\
\item Referring to Lemma \ref{nearfielddecomp}, $\oplus_{i=1}^{s} A v_r$ is a subspace of ${\cal V}_j$ where $v_r \in {\cal V}_j$ for $r \in \{1,\ldots,s\}.$  
\end{itemize}
\end{corollary}

\begin{example}
Consider the near-vector space $({\mathbb{R}}^{3},\mathbb{R}),$ where scalar multiplication is defined for all $(x,y,z) \in{\mathbb{R}}^{3}$ and $\alpha \in \mathbb{R}$ by 
$$\alpha(x,y,z) = (\alpha^{3}x, \alpha^{5}y,\alpha^{3}z).$$
Then by applying Corollary \ref{decomposition}, we have that 
$${\cal V}_1 =\{(a,b,c) \in {\mathbb{R}}^{3} |\forall \alpha , \beta \in \mathbb{R}, \alpha+_{(a,b,c)} \beta = (\alpha^{3}+\beta^{3})^{\frac{1}{3}} =\{(a,0,c)|a,c \in {\mathbb{R}}\},$$ while 
$${\cal V}_2 =\{(a,b,c) \in {\mathbb{R}}^{3} |\forall \alpha , \beta \in \mathbb{R},\alpha +_{(a,b,c)} \beta = (\alpha^{5}+\beta^{5})^{\frac{1}{5}}\} =\{(0,b,0)| b \in {\mathbb{R}}\}$$ and
${\mathbb{R}}^{3} = {\cal V}_1 \oplus {\cal V}_2.$ 
\end{example}

\subsection{The structure of Span} 

It is immediate to prove the following lemmas:

\begin{lemma} \label{dependent} 
Let $(V,A)$ be a near-vector space. 

$v_0, \cdots, v_n\in Q(V)$ are linearly dependent if there is $i_0 \in \{ 1, \cdots , n\}$ such that $v_{i_0}\in span ( v_1, \cdots, v_n)$. 
\end{lemma} 

We have explicitly seen in Lemma \ref{nearfielddecomp} how $V \simeq A^{(I)}$ allows us to retranslate the following lemma to any near-vector space.

\begin{theorem} \label{spanstruc} 
Let $V=A^{(I)}$ as in Lemma \ref{nearfielddecomp}, 2. and $(A, +_v, \cdot)$ be a division ring for any nonzero $v \in V$. Let $v \in V.$

According to Corollary \ref{decomposition}, we have the regular decomposition $V= \oplus_{j \in K} {\cal V}_j$, so that we can write $v$ uniquely as $v = \sum_{j \in K} v_i$ where $v_i \in {\cal V}_i$. If $N$ is the subset of $K$, $N = \{ i \in K | v_i \neq 0\},$
we have that 
$$span (v)=L(v)= \oplus A v_i$$ 
and 
$$ dim (v) = |N|.$$ 
\end{theorem}
\begin{proof}
Let $v$ be as in the statement, i.e. 
$$ v= \sum_{i\in N} v_i.$$ 
For convenience we order the elements in N as $\{i_1, \ldots , i_n\}$ where $n = |N|$. 

Since the $v_i$ are in different regular components and we are assuming the division ring condition, using Theorem \ref{vstheorem}, we know that $+_{v_i} \neq +_{v_j}$ for any $i\neq j \in N$. 

Therefore there exist $\alpha , \beta\in A$ such that 
$$ \alpha +_{v_{i_1}} \beta \neq \alpha +_{v_{i_2}} \beta.$$ 

Then, taking 
$$w= (\alpha v + \beta v) -(\alpha +_{v_{i_2}} \beta) v \in L(v) \subseteq span(v),$$ 

we have that $w= \sum_{i\in N_1} w_i$ 
where $N_1\subseteq N\backslash \{ i_2 \},$ $i_1 \in N_1$ and $w_i \in {\cal V}_i$ nonzero. 

Repeating this process successively, we can construct an element $\omega_1= \theta_1 v_{i_1} \in L(v) = span(v) $ and more generally restarting from $v$ we can construct $\omega_j = \theta_j v_{i_j} \in L(v) = span(v)$. 

Since $ A\theta_j v_{i_j} = A v_{i_j},$ this proves that
$$ \oplus_{i\in N} A v_i \subseteq span(v).$$ 
But since $ \oplus_{i\in N} A v_i$ is a near-vector space contained in $L(v)$, we therefore have $\oplus_{i\in N} A v_i=L(v)$ and since it is generated by elements of $Q(V)$, we get the result
$$ \oplus_{i\in N} A v_i=L(v) = span(v).$$ 
\end{proof} 

We now give an example to illustrate the previous result.

\begin{example}
Consider the near-vector space $(({{\mathbb{Z}}_{11}})^{3},{\mathbb{Z}}_{11}),$ where scalar multiplication is defined for all $(x,y,z) \in{\mathbb{R}}^{3}$ and $\alpha \in \mathbb{R}$ by 
$$\alpha(x,y,z) = (\alpha^{3}x, \alpha^{5}y,\alpha^{3}z).$$
Note that $$Q(V) = \{(a,0,c)|a,c, \in {\mathbb{Z}}_{11}\}\cup\{(0,b,0)|b \in {\mathbb{Z}}_{11}\}.$$
Take for example, $v = (2,5,6) \in ({{\mathbb{Z}}_{11}})^{3},$ then $v \notin Q({\mathbb{Z}}_{11}),$ and
$$span((2,5,6)) = {\mathbb{Z}}_{11}(2,0,6) + {\mathbb{Z}}_{11}(0,5,0).$$
If we take, for example $w = (3,0,4),$ then $w \in Q({\mathbb{Z}}_{11}),$ and 
$$span((3,0,4)) = {\mathbb{Z}}_{11}(3,0,4).$$
\end{example}

An easy but useful consequence of the previous theorem is the following. 

\begin{corollary}
Let $(V,A)$ be a near-vector space. 
For every $S= \{ v_i \in Q(V) | i \in I\},$ where $I$ is an index set, 
$$ span(S) =L( S)= \oplus_{t\in T}  A v_{i_t} $$
where $\{ v_{i_t},t\in T\} $ is the biggest linearly independent subset of $S$. 
\end{corollary} 
More generally we have, 
\begin{corollary}
Let $(V,A)$ be a non-regular near-vector space and suppose $(A, +_v , \cdot)$ is a division ring for all $v\in Q(V)\backslash\{0\}$. Let $S= \{ v_i \in V | i \in I\}$ where $I$ is an index set, possibly infinite. Then we have that
$$span(S) = L(S) = \oplus_{i \in T} A\omega_i,$$
where $\omega_i\in Q(V)$, for all $i \in T$. 
\end{corollary}
\begin{proof} 
For any $j\in I$, 
$$ v_j =\sum_{i \in N_j}  {v_j}_i $$
according to the regular decomposition where ${v_j}_i \in {\cal V}_{i}$ nonzero. 

It is not hard to see that 
$$span(v_j)=\oplus_{i \in N_j}  A{v_j}_i\subseteq span(S),$$
where
$$N_j = \{ i \in I | {v_j}_i  \neq 0 \}.$$ 
Therefore, 
$$ \sum_{i\in I } span(v_j) \subseteq span(S).$$ 
From the previous corollary we know that 
$$ span ( \{ {v_j}_i| j \in I\}) = \oplus_{t\in T_{{i, j}} }A {w_{i,j}} _t,$$
where $\{{w_{i,j}} _t | t\in T_{i, j} \}  $ is the biggest linearly independent subset of $\{ {v_j}_i| j \in I\}$. 

We leave the details to the reader to conclude that 
$$span(S) = L(S) = \oplus_{j \in I, i \in N_j, t \in T_{i, j} } A{w_{i,j}} _t,$$
proving the result. 
\end{proof}

From the previous corollary, we can rectify an error made in Lemma 2.4 in \cite{Howell2}, p.2426, where it was shown that a subset of a near-vector space is a subspace if and only if it is closed under addition and scalar multiplication. Of course, the one direction is obvious. The problem with the converse was that the proposed generating set was not necessarily contained in the subspace.
We rectify it here for the case where we assume $+_v = +_w$ for any $v, w \in Q(V)\backslash\{0\}.$ 

\begin{corollary}
Let $(V,A)$ be a non-regular near-vector space and suppose $(A, +_v , \cdot)$ is a division ring for all nonzero $v\in Q(V)$. $W$ is a subspace of $V$ if and only it is non-empty, closed under addition and scalar multiplication. 
\end{corollary}

\begin{corollary} 
Let $(V,A)$ be a non-regular near-vector space and suppose $(A, +_v , \cdot)$ is a division ring for all $v\in Q(V)\backslash\{0\}$. 

There exist $v$ and $w \in V\backslash Q(V)$ $v \neq w$ such that $span (v) = span (w)$. 
\end{corollary} 

\begin{proof} 
Take $v_1, v_2 \in Q(V)$ linearly independent and not in the same regular component and $v = v_1 + v_2$ and $w= \theta v_1 + v_2 \in  V\backslash Q(V)$, where $\theta \neq 1$, $v\neq w$, but  $span (v) = span (w)$. 
\end{proof}

The following two corollaries shed more light on why taking elements outside of $Q(V)$ as basis elements would be counter-intuitive to our general intuition with respect to a basis.

\begin{corollary} 
Let $(V,A)$ be a non-regular near-vector space with $dim (V)> 2$ and suppose $(A, +_v , \cdot)$ is a division ring for all $v\in Q(V)\backslash\{0\}$. 
Then there exisst $v$ and $w \in V\backslash Q(V)$ such that $v \notin span (w)$ and $w\notin span (v)$ and $span (v) \cap span (w) \neq \{ 0\}$. 
\end{corollary} 
\begin{proof} 
Take $v_1, v_2,v_3 \in Q(V)$ linearly independent not in the same regular subspace in the decomposition of $V$ and $v = v_1 + v_2$ and $w= v_2 + v_3 \in  V\backslash Q(V)$, then we have  $v \notin span (w)$ and $w\notin span (v),$ but 
$$span (v) \cap span (w)= Av_2\neq \{ 0 \}.$$ 
\end{proof}

\section*{Acknowledgment}
The authors would like to express their gratitude for funding by the National Research Foundation (Grant number: 93050) and Stellenbosch University. The authors would like to thank Jacques Rabie for Example \ref{jacques} and Charlotte Kestner for Example \ref{lotte}.

\end{document}